\documentclass[10pt]{amsart}
\usepackage{amsmath, amssymb}
\usepackage{mathrsfs}
\newcommand{\no}[1]{#1}
\renewcommand{\no}[1]{}
\no{\usepackage{times}\usepackage[subscriptcorrection, slantedGreek, nofontinfo]{mtpro}
\renewcommand{\Delta}{\upDelta}}
\usepackage{color}

\date{\today}
\newtheorem{theorem}{Theorem}[section]

\newtheorem{corollary}{Corollary}[section]

\theoremstyle{remark}
\newtheorem{remark}{Remark}[section]

\numberwithin{equation}{section}

\title[Determining the potential in a wave equation]{Determining the potential in a wave equation without a geometric condition. Extension to the heat equation}

\author[Ka\"{\i}s Ammari]{Ka\"{\i}s Ammari}
\address{UR Analysis and Control of Pde, UR 13ES64, Department of Mathematics, Faculty of Sciences of Monastir, University of Monastir, 5019 Monastir, Tunisia }
\email{kais.ammari@fsm.rnu.tn}

\author[Mourad Choulli]{Mourad Choulli}
\address{Institut \'Elie Cartan de Lorraine, UMR CNRS 7502, Universit\'e de Lorraine, Boulevard des Aiguillettes, BP 70239, 54506 Vandoeuvre les Nancy cedex - Ile du Saulcy, 57045 Metz cedex 01, France}
\email{mourad.choulli@univ-lorraine.fr}

\author[Faouzi Triki]{Faouzi Triki\dag}
\address{Laboratoire Jean Kuntzmann, UMR CNRS 5224, Universit\'e de Joseph Fourier, 38041 Grenoble Cedex 9,
France}
\email{Faouzi.Triki@imag.fr}

\date{}

\begin{document}

\begin{abstract}
We prove a logarithmic stability estimate for the inverse problem of determining the potential in a wave equation from boundary measurements obtained by varying the first component of the initial condition. The novelty of the present work is that no geometric condition is imposed to the sub-boundary where the measurements are made. Our results improve those obtained by the first and second authors in \cite{AC1}. We also show how the analysis for the wave equation can be adapted to  an inverse coefficient problem for the heat equation.

\end{abstract}

\subjclass[2010]{35R30}

\keywords{inverse problem, wave equation, potential, boundary measurements. \\ \dag FT is partially supported by Labex PERSYVAL-Lab (ANR-11-LABX-0025-01)}

\maketitle


\section{Introduction}

Let $\Omega$ be a $C^3$-smooth bounded domain of $\mathbb{R}^n$, $n\ge 2$, with boundary $\Gamma$ and consider the following initial-boundary value problem, abbreviated to IBVP in the sequel, for the wave equation:
\begin{equation}\label{1.1}
\left\{
\begin{array}{lll}
 \partial _t^2 u - \Delta u + q(x)u  = 0 \;\; &\mbox{in}\;   Q=\Omega \times (0,\tau), 
 \\
u = 0 &\mbox{on}\;  \Sigma =\Gamma \times (0,\tau), 
\\
u(\cdot ,0) = u_0,\; \partial_t u (\cdot ,0) = u_1.
\end{array}
\right.
\end{equation}

From here on
\[
E_0=H_0^1(\Omega )\oplus L^2(\Omega ).
\]

The unit ball of a Banach space $X$ will denoted in the sequel by $B_X$.

\smallskip
By \cite[Theorem A.3, page 493]{BCY}, for any $(u_0,u_1)\in E_0$ and $q\in L^\infty (\Omega )$, the IBVP \eqref{1.1} has a unique solution
 \[u:=\mathscr{S}_q^\tau (u_0,u_1)\in C([0,\tau ]; H_0^1(\Omega ))\] 
 so that 
 \[ \partial _t u\in C([0,\tau ]; L^2(\Omega ))\;\; \mathrm{and}\;\; \partial _\nu u\in L^2(\Sigma ).\] 
 Additionally, for any $m>0$, there exists a constant $C=C(m,\Omega )>0$ so that, for each $q\in mB_{L^\infty (\Omega )}$ and $(u_0,u_1)\in E_0$,
\[
\|\partial _\nu \mathscr{S}_q^\tau (u_0,u_1)\|_{L^2(\Sigma )}\leq C\|(u_0,u_1)\|_{E_0}.
\]

Fix $\Upsilon$ a non empty open subset of $\Gamma$ and set $\Lambda =\Upsilon \times (0,\tau )$. The inequality above says that the operator
\[
\mathscr{C}_q^\tau : (u_0,u_1)\in E_0\mapsto \partial _\nu \mathscr{S}_q^\tau (u_0,u_1)_{|\Lambda}\in L^2(\Lambda )
\] 
is bounded and
\begin{equation}\label{1.4}
\| \mathscr{C}_q^\tau\|_{\mathscr{B}(E_0,L^2(\Lambda ))}\le C,
\end{equation}
uniformly in $q\in mB_{L^\infty (\Omega )}$.

\smallskip
Define the operator $\widetilde{\mathscr{C}}_q^\tau $ by $\widetilde{\mathscr{C}}_q^\tau (u_0)=\mathscr{C}_q^\tau (u_0,0)$, $u_0\in H_0^1(\Omega )$. Clearly $\widetilde{\mathscr{C}}_q^\tau \in \mathscr{B}(H_0^1(\Omega ), L^2(\Lambda ))$ and 
\begin{equation}\label{1.5}
\| \widetilde{\mathscr{C}}_q^\tau \|_{\mathscr{B}(H_0^1(\Omega ),L^2(\Lambda ))}\le C,
\end{equation}
again uniformly in $q\in mB_{L^\infty (\Omega )}$.

\smallskip
Recall that the space $H_\Delta (\Omega )$ is given by
\[
H_\Delta (\Omega )=\{ \varphi \in L^2(\Omega );\; \Delta \varphi \in L^2(\Omega )\}.
\]
With reference to Poincar\'e's inequality,  $\mathcal{H}=H_0^1(\Omega )\cap H_\Delta (\Omega )$  is a Banach space for the norm
\[
\| \varphi \|_{\mathcal{H}}=\|\nabla \varphi \|_{L^2(\Omega )^n}+\|\Delta \varphi \|_{L^2(\Omega )}.
\]
When $u_0\in \mathcal{H}$ we easily see that $\partial _t\mathscr{S}_q^\tau (u_0,0)=\mathscr{S}_q^\tau (0,\Delta u_0-qu_0)$. Then proceeding similarly as before we conclude that $\widetilde{\mathscr{C}}_q^\tau $ restricted to $\mathcal{H}$, still denoted by $\widetilde{\mathscr{C}}_q^\tau $,  define a bounded operator from $\mathcal{H}$ into $H^1((0,\tau) ;L^2(\Upsilon ))$ and 
\begin{equation}\label{1.6}
\| \widetilde{\mathscr{C}}_q^\tau \|_{\mathscr{B}(\mathcal{H},H^1((0,\tau) ;L^2(\Upsilon ))}\le C,
\end{equation}
uniformly in $q\in mB_{L^\infty (\Omega )}$.

\smallskip
Let
\[
\Psi (\gamma )=|\ln \gamma |^{-\frac{1}{8+2n}} +\gamma ,\;\; \gamma >0,
\]
extended by continuity at $\gamma =0$ by setting $\Psi (0)=0$.

\smallskip
Let $m>0$ be fixed. In the rest of this text, unless otherwise stated, $C$, $c$, $\tau _0$ and $\mu$ denote generic constants that can depend only on $n$, $\Omega$, $\Upsilon$ and $m$.

\smallskip
We aim to prove the following theorem.

\begin{theorem}\label{theorem1.1}
There exist two constants $\tau_0 >0$ and $C>0$ so that, for any $\tau \geq \tau_0$, $q_0,q\in mB_{L^\infty (\Omega )}$ satisfying $q_0\geq 0$ and $q-q_0\in mB_{W^{1,\infty }(\Omega )}$,
\[
C\|q_0-q\|_{L^2(\Omega )}\le \Psi \left(\| \widetilde{\mathscr{C}}_q^\tau -\widetilde{\mathscr{C}}_{q_0}^\tau \|_{\mathscr{B}(\mathcal{H},H^1((0,\tau) ;L^2(\Upsilon ))}\right).
\]
\end{theorem}

The proof of Theorem \ref{theorem1.1} we present here follows the method initiated by the first and second authors in \cite{AC1}. This method is mainly based on a spectral decomposition combined with an observability inequality. Due to the fact that we do not assume any geometric condition on the domain, the classical observability inequality is no longer valid in our case. We substitute it by an interpolation inequality established by Robbiano in \cite{Ro}. It is worthwhile to mention that a spectral decomposition combined with an observability inequality was also used in \cite{AC2} to establish a logarithmic stability estimate for the problem of determining a boundary coefficient in a wave equation from boundary measurements.

\smallskip
To our knowledge, using observability inequalities to solve inverse problems related to the wave equation  goes back to Puel and Yamamoto \cite{PY}. Later, Komornik and Yamamoto \cite{KY} applied this method to an inverse point source problem for a wave equation. A general framework of this method is due to Alves, Silvestre, Takahashi and Tucsnak \cite{ASTT} and extended recently to singular sources by Tucsnak and Weiss \cite{TW}.

\smallskip
The rest of this paper consists in two sections. Section 2 is devoted to the proof of Theorem \ref{theorem1.1}. In Section 3, we adapt our approach to an inverse problem for the heat equation.


\section{Proof of Theorem \ref{theorem1.1}}

We firstly observe that a careful examination of the proof of \cite[Theorem 1, page 98]{Ro} allows us to deduce the following result.

\begin{theorem}\label{theorem1.2}
There exist three constants $\tau_0>0$, $C>0$ and $\mu >0$ so that, for all $\tau \ge \tau_0$, $(u_0,u_1)\in E_0$, $q\in mB_{L^\infty (\Omega )}$  and $\epsilon >0$,
\[
C\|(u_0,u_1)\|_{E_{-1}}\le \frac{1}{\sqrt{\epsilon}}\|(u_0,u_1)\|_{E_0}+e^{\mu \epsilon}\|\mathscr{C}_q^\tau (u_0,u_1)\|_{L^2(\Lambda )},
\]
where $E_{-1}=L^2(\Omega )\oplus H^{-1}(\Omega )$.
\end{theorem}

From now on, $\tau \ge \tau_0$ is fixed, where $\tau_0$ is as in the preceding theorem.

\smallskip
Let $g \in H^1((0,\tau ))$ satisfying $g(0)\ne 0$ and consider the IBVP
\begin{equation}\label{1.2}
\left\{
\begin{array}{lll}
 \partial _t^2 v - \Delta v + q(x)v  = g(t)f(x) \;\; &\mbox{in}\;   Q, 
 \\
v = 0 &\mbox{on}\;  \Sigma , 
\\
v(\cdot ,0) = 0,\; \partial_t v (\cdot ,0) = 0,
\end{array}
\right.
\end{equation}

From  \cite[Theorem A.3, page 493]{BCY},  the IBVP \eqref{1.2} has a unique solution 
\[v:=\mathcal{S}_q^\tau (f,g)\in C([0,\tau ]; H_0^1(\Omega ))\]
so that 
\[\partial _t v\in C([0,\tau ]; L^2(\Omega ))\;\; \mathrm{and}\;\; \partial _\nu v\in L^2(\Sigma ).\] 
Applying Duhamel's formula we get in a straightforward manner
\[
\mathcal{S}_q^\tau (f,g)(\cdot ,t)=\int_0^t g(t-s)\mathscr{S}_q^\tau (0,f)(\cdot ,s)ds.
\]
Therefore
\[
\mathcal{C}_q^\tau (f,g)(\cdot ,t):=\partial _\nu \mathcal{S}_q^\tau (f,g)(\cdot ,t)=\int_0^t g(t-s)\mathscr{C}_q^\tau (0,f)(\cdot ,s)ds.
\]

Let
\[
H^1_\ell ((0,\tau), L^2(\Upsilon )) = \left\{u \in H^1((0,\tau), L^2(\Upsilon )); \; u(0) = 0 \right\}
\]

and define the operator $S:L^2(\Lambda )\longrightarrow H^1_\ell ((0,\tau ) ,L^2(\Upsilon ))$ by
\begin{equation*}
(Sh)(t)=\int_0^t g(t-s)h(s)ds.
\end{equation*}

From \cite[Theorem 2.1]{AC1} and its proof, $S$ is an isomorphism and
\[
\|h\|_{L^2(\Lambda )}\leq \frac{\sqrt{2}}{|g(0)|}e^{\tau\frac{\|g'\|_{L^2((0,\tau))}^2}{|g(0)|^2}}\|Sh \|_{H^1((0,\tau ),L^2(\Upsilon ))}.
\]
Consequently
\[
\|\mathscr{C}_q^\tau (0,f)\|_{L^2(\Lambda )}\le \frac{\sqrt{2}}{|\lambda (0)|}e^{\tau\frac{\|\lambda '\|_{L^2((0,\tau))}^2}{|\lambda (0)|^2}}\| \mathcal{C}_q^\tau (f,g)\|_{H^1 ((0,\tau ),L^2(\Upsilon ))}.
\]
In combination with the estimate in Theorem \ref{theorem1.2}, this inequality yields, where $\epsilon >0$ is arbitrary,
\begin{equation}\label{1.3}
C\|f\|_{H^{-1}(\Omega )}\le \frac{1}{\sqrt{\epsilon}}\|f\|_{L^2(\Omega )}+\frac{\sqrt{2}}{|g(0)|}e^{\tau\frac{\|g'\|_{L^2((0,\tau))}^2}{|g(0)|^2}}e^{\mu \epsilon}\| \mathcal{C}_q^\tau (f,g)\|_{H^1 ((0,\tau ),L^2(\Upsilon ))}.
\end{equation}

Let $q_0,q\in mB_{L^\infty (\Omega )}$ satisfying $q_0\geq 0$ and $q-q_0\in mB_{W^{1,\infty }(\Omega )}$.

\smallskip
 Consider the unbounded operator $A_0:L^2(\Omega )\rightarrow L^2(\Omega )$ given by
\[
A_0=-\Delta +q_0 ,\;\;  D(A_0)=H_0^1(\Omega )\cap H^2(\Omega ).
\]
Let $0 <\lambda _1\le \lambda _2 \le  \ldots \le \lambda _k \ldots \rightarrow +\infty$ be the sequence of eigenvalues of  the operator $A_0$ and $(\phi _k)$ a sequence of the corresponding eigenfunctions so that $(\phi _k)$ form an orthonormal basis of $L^2(\Omega )$.

\smallskip
It is a simple exercise to check that $\mathscr{S}_{q_0}^\tau (\phi _k,0)=g_k(t)\phi _k$ with $g_k(t)=\cos (\sqrt{\lambda_k}t)$.

Observing that 
\[
\mathscr{S}_q^\tau (\phi _k,0)-\mathscr{S}_{q_0}(\phi _k,0)=\mathcal{S}_q^\tau ((q-q_0)\phi _k,g_k),
\]
we get
\[
\widetilde{\mathscr{C}}_q^\tau (\phi _k)-\widetilde{\mathscr{C}}_{q_0}^\tau (\phi _k)=\mathcal{C}_q^\tau ((q-q_0)\phi _k,g_k).
\]
Hence \eqref{1.3} gives

\[
C\|(q-q_0)\phi _k\|_{H^{-1}(\Omega )}\le \frac{1}{\sqrt{\epsilon}}\|(q-q_0)\phi _k\|_{L^2(\Omega )}+e^{\tau^2\lambda _k}e^{\mu \epsilon}\| \widetilde{\mathscr{C}}_q^\tau (\phi _k)-\widetilde{\mathscr{C}}_{q_0}^\tau (\phi _k)\|_{H^1_\ell ((0,\tau ),L^2(\Upsilon ))}.
\]
This and the fact that $\|\phi _k\|_{\mathcal{H}}\le \sqrt{\lambda _k}+m+\lambda _k$ imply
\[
C\|(q-q_0)\phi _k\|_{H^{-1}(\Omega )}\le \frac{1}{\sqrt{\epsilon}}\|(q-q_0)\phi _k\|_{L^2(\Omega )}+(\sqrt{\lambda _k}+m+\lambda _k)e^{\tau^2\lambda _k}e^{\mu \epsilon}\| \widetilde{\mathscr{C}}_q^\tau -\widetilde{\mathscr{C}}_{q_0}^\tau \|_{\mathscr{B}(\mathcal{H}, H^1((0,\tau ),L^2(\Upsilon )))}.
\]
But $(\sqrt{\lambda _k}+m+\lambda _k) \leq (\mu _1^{-1/2}+m\mu^{-1}+1)\lambda _k\le e^{(\mu _1^{-1/2}+m\mu^{-1}+1)\lambda _k}$, where $\mu _1$ is the first eigenvalue of the Laplace operator under Dirichlet boundary condition. Whence
\[
C\|(q-q_0)\phi _k\|_{H^{-1}(\Omega )}\le \frac{1}{\sqrt{\epsilon}}\|(q-q_0)\phi _k\|_{L^2(\Omega )}+e^{\kappa \lambda _k}e^{\mu \epsilon}\| \widetilde{\mathscr{C}}_q^\tau -\widetilde{\mathscr{C}}_{q_0}^\tau \|_{\mathscr{B}(\mathcal{H}, H^1((0,\tau ),L^2(\Upsilon )))}.
\]
Here $\kappa =\tau ^2+\mu _1^{-1/2}+m\mu^{-1}+1$.

\smallskip
Since $\|(q-q_0)\phi _k\|_{L^2(\Omega )}\le m$, we have
\[
C\|(q-q_0)\phi _k\|_{H^{-1}(\Omega )}\le \frac{1}{\sqrt{\epsilon}}+e^{\kappa \lambda _k}e^{\mu \epsilon}\| \widetilde{\mathscr{C}}_q^\tau -\widetilde{\mathscr{C}}_{q_0}^\tau \|_{\mathscr{B}(\mathcal{H}, H^1((0,\tau ),L^2(\Upsilon )))}.
\]

Using the usual interpolation inequality 
\[
\|h\|_{L^2(\Omega )}\le c\|h\|_{H_0^1(\Omega )}^{1/2}\|h\|_{H^{-1}(\Omega )}^{1/2}\;\; h\in H_0^1(\Omega ),
\]
we obtain
\[
C\|(q-q_0)\phi _k\|_{L^2(\Omega )}^2\le \|(q-q_0)\phi _k\|_{H_0^1(\Omega )}\left( \frac{1}{\sqrt{\epsilon}}+e^{\kappa \lambda _k}e^{\mu \epsilon}\| \widetilde{\mathscr{C}}_q^\tau -\widetilde{\mathscr{C}}_{q_0}^\tau \|_{\mathscr{B}(\mathcal{H}, H^1((0,\tau ),L^2(\Upsilon )))}\right).
\]

Bearing in mind that $\| q-q_0\|_{W^{1,\infty}(\Omega )}\le m$, we get
\begin{align*}
\|(q-q_0)\phi _k\|_{H_0^1(\Omega )}&\le \|(q-q_0)\nabla \phi _k\|_{L^2(\Omega )^n}+\|\phi _k\nabla (q-q_0)\|_{L^2(\Omega )^n}
\\
&\le m\sqrt{\lambda_k}+m
\\
&\le m(1+\mu_1^{-1/2})\sqrt{\lambda_k}
\\
&\le m\mu_1^{-1/2}(1+\mu_1^{-1/2})\lambda_k .
\end{align*}

Consequently
\[
C\|(q-q_0)\phi _k\|_{L^2(\Omega )}\le \frac{\lambda _k}{\sqrt{\epsilon}}+e^{(\kappa +1)\lambda _k}e^{\mu \epsilon }\| \widetilde{\mathscr{C}}_q^\tau -\widetilde{\mathscr{C}}_{q_0}^\tau \|_{\mathscr{B}(\mathcal{H}, H^1((0,\tau ),L^2(\Upsilon )))}.
\]

\smallskip
Denote the scalar product of $L^2(\Omega )$ by $(\cdot ,\cdot )_{L^2(\Omega )}$. By Cauchy-Schwarz's inequality 
\[\left| (q-q_0,\phi _k)_{L^2(\Omega )}\right|\leq |\Omega|^{1/2}\|(q-q_0)\phi _k\|_{L^2(\Omega )}.\]
Therefore
\begin{equation}\label{1.7}
C (q-q_0,\phi _k)_{L^2(\Omega )}^2\le  \frac{\lambda _k}{\sqrt{\epsilon}}+e^{(\kappa +1)\lambda _k}e^{\mu \epsilon }\| \widetilde{\mathscr{C}}_q^\tau -\widetilde{\mathscr{C}}_{q_0}^\tau \|_{\mathscr{B}(\mathcal{H}, H^1((0,\tau ),L^2(\Upsilon )))}.
\end{equation}

According to the min-max principle, there exists $\widetilde{c}>1$ (depending on $m$ but not  on $q$) so that
\begin{equation}\label{1.8}
\widetilde{c}^{-1}k^{2/n}\le \lambda _k\le \widetilde{c}k^{2/n}.
\end{equation}
We refer to \cite{Ka} for a proof.

\smallskip
Estimates  \eqref{1.7} and \eqref{1.8} entail
\begin{equation}\label{1.9}
C (q-q_0,\phi _k)_{L^2(\Omega )}^2\le  \frac{k^{2/n}}{\sqrt{\epsilon}}+e^{\varrho k^{2/n}}e^{\mu \epsilon }\| \widetilde{\mathscr{C}}_q^\tau -\widetilde{\mathscr{C}}_{q_0}^\tau \|_{\mathscr{B}(\mathcal{H}, H^1((0,\tau ),L^2(\Upsilon )))},
\end{equation}
with $\varrho =\widetilde{c}(\kappa +1)$.

\smallskip
Let $N\geq 1$ be an integer. Using that $\left(\sum_{k\ge 1}(1+\lambda _k)(\cdot ,\phi _k)_{L^2(\Omega )}^2\right)^{1/2}$ is an equivalent norm on $H^1(\Omega )$, 
\[
\|q-q_0\|_{H^1(\Omega )}^2\le |\Omega |\left(\|q-q_0\|_{L^\infty(\Omega )}^2+\|\nabla (q-q_0)\|_{L^\infty(\Omega )^n}^2\right)\le c_\Omega \|q-q_0\|_{W^{1,\infty}(\Omega )}^2
\]
and \eqref{1.8}, we get
\begin{align*}
\|q-q_0\|_{L^2(\Omega )}^2 &=\sum_{k\le N}(q-q_0,\phi _k)_{L^2(\Omega )}^2+\sum_{k>N}(q-q_0,\phi _k)_{L^2(\Omega )}^2
\\
&\le \sum_{k\le N}(q-q_0,\phi _k)_{L^2(\Omega )}^2 +\frac{1}{\lambda _{N+1}}\sum_{k>N}\lambda _k(q-q_0,\phi _k)_{L^2(\Omega )}^2
\\
&\le \sum_{k\le N}(q-q_0,\phi _k)_{L^2(\Omega )}^2 +\frac{c_\Omega m^2}{\widetilde{c}(N+1)^{2/n}}.
\end{align*}
In combination with \eqref{1.9}, this estimate yields
\begin{equation}\label{1.10}
C\|q-q_0\|_{L^2(\Omega )}^2 \le \frac{N^{1+2/n}}{\sqrt{\epsilon}}+\frac{1}{(N+1)^{2/n}}+Ne^{\varrho N^{2/n}}e^{\mu \epsilon }\| \widetilde{\mathscr{C}}_q^\tau -\widetilde{\mathscr{C}}_{q_0}^\tau \|_{\mathscr{B}(\mathcal{H}, H^1((0,\tau ),L^2(\Upsilon )))}.
\end{equation}

For $s\geq 1$ a real number, let $N$ be the unique integer so that $N\leq s<N+1$. Then \eqref{1.10} implies
\[
C\|q-q_0\|_{L^2(\Omega )}^2 \le \frac{s^{1+2/n}}{\sqrt{\epsilon}}+\frac{1}{s^{2/n}}+se^{\varrho s^{2/n}}e^{\mu \epsilon }\| \widetilde{\mathscr{C}}_q^\tau -\widetilde{\mathscr{C}}_{q_0}^\tau \|_{\mathscr{B}(\mathcal{H}, H^1((0,\tau ),L^2(\Upsilon )))}.
\]
Taking $\epsilon =s^{8/n+2}$ in this inequality, we find
\[
C\|q-q_0\|_{L^2(\Omega )}^2 \le \frac{1}{s^{2/n}}+se^{\varrho s^{2/n}}e^{\mu s^{8/n+2}}\| \widetilde{\mathscr{C}}_q^\tau -\widetilde{\mathscr{C}}_{q_0}^\tau \|_{\mathscr{B}(\mathcal{H}, H^1((0,\tau ),L^2(\Upsilon )))},
\]
and then
\begin{equation}\label{1.11}
C\|q-q_0\|_{L^2(\Omega )}^2 \le \frac{1}{s^{2/n}}+e^{\theta s^{8/n+2}}\| \widetilde{\mathscr{C}}_q^\tau -\widetilde{\mathscr{C}}_{q_0}^\tau \|_{\mathscr{B}(\mathcal{H}, H^1((0,\tau ),L^2(\Upsilon )))},
\end{equation}
with $\theta =1+\varrho+\mu$.

\smallskip
We use the temporary notation $\gamma =\| \widetilde{\mathscr{C}}_q^\tau -\widetilde{\mathscr{C}}_{q_0}^\tau \|_{\mathscr{B}(\mathcal{H}, H^1((0,\tau ),L^2(\Upsilon )))}$. We consider the function $\chi (s)=s^{2/n}e^{\theta s^{8/n+2}}$, $s\geq 1$. Under the condition $\gamma \le \gamma ^\ast =e^{-\theta}$, there exist $s^\ast \ge 1$ so that $\chi (s^\ast)=\gamma ^{-1}$. In that case $s=s^\ast$ in \eqref{1.11} gives in a straightforward manner
\begin{equation}\label{1.12}
C\|q-q_0\|_{L^2(\Omega )} \leq |\ln \gamma |^{-\frac{1}{8+2n}}.
\end{equation}
When $\gamma \ge \gamma ^\ast$, we have trivially
\begin{equation}\label{1.13}
\|q-q_0\|_{L^2(\Omega )}\le m |\Omega |^{1/2}\leq m |\Omega |^{1/2}\frac{\gamma}{\gamma ^\ast}.
\end{equation}
In light of \eqref{1.12} and \eqref{1.13}, we end up getting
\[
C\|q-q_0\|_{L^2(\Omega )} \le \Psi \left(\| \widetilde{\mathscr{C}}_q^\tau -\widetilde{\mathscr{C}}_{q_0}^\tau \|_{\mathscr{B}(\mathcal{H}, H^1((0,\tau ),L^2(\Upsilon )))}\right)
\]
as it is expected.

\begin{remark}\label{remark1.1}
Fix $g$ and $q$ in \eqref{1.2}, and let $\mathcal{C}^\tau (f) :=\mathcal{C}_q^\tau (f ,g)$. We can then use \eqref{1.3} to derive a stability estimate for inverse source problem consisting in the determination of $f$ from $\mathcal{C}^\tau (f)$. A minimization argument in $\epsilon$ leads to the following result: there exist two constants $\tau _0>0$ and  $C>0$ so that, for any $\tau \ge \tau_0$ and $f\in mB_{L^2(\Omega )}$,
\[
C\|f\|_{H^{-1}(\Omega )}\le \Phi \left(\| \mathcal{C}^\tau (f) \|_{H^1((0,\tau ),L^2(\Upsilon ))}\right),
\]
where $\Phi (\gamma )=\left| \ln \gamma \right|^{-1/2}+\gamma$, $\gamma >0$.
\end{remark}

\begin{remark}\label{remark1.2}
Let us explain briefly how one can get a stability estimate for the inverse problem of determining the damping coefficient or the damping coefficient together with the potential, from boundary measurements, again by varying the initial conditions. Let us substitute in the first equation of the IBVP \eqref{1.1} the operator $W=\partial _t^2 -\Delta$ by $W_a=\partial _t^2-\Delta +a(x)\partial _t$. The function $a$ is usually taken bounded and non negative. It is called the damping coefficient. In that case Theorem \ref{theorem1.2} holds if the operator $W$ is substituted by  $W_a$. This can be seen by examining the proof in \cite{Ro}. The only difference between the two cases is that in the present case there is an additional term in the quantity between the brackets in \cite[formula (8), page 109]{Ro}. But this supplementary term does not modify the estimate  \cite[formula (9), page 109]{Ro}. The rest of the analysis remains unchanged. These observations at hand, we can extend \cite[Theorem 1.1, (1.3) and (1.4)]{AC1} and \cite[Theorem 4.2]{AC1} to the case where no geometric condition is imposed to the  sub-boundary where the measurements are made. We leave to the interested reader to write down the details.
\end{remark}


\section{Extension to the heat equation}

We begin with an inverse source problem associated to the following IBVP for the heat equation.
\begin{equation}\label{3.1}
\left\{
\begin{array}{lll}
 \partial _t u - \Delta u + q(x)u  = g(t)f(x) \;\; &\mbox{in}\;   Q, 
 \\
u = 0 &\mbox{on}\;  \Sigma, 
\\
u(\cdot ,0) = 0.
\end{array}
\right.
\end{equation}

From now on $\tau >0$ is arbitrary but fixed.

\smallskip
Recall that the anisotropic Sobolev space $H^{2,1}(Q)$ is given as follows
\[
H^{2,1}(Q)=L^2((0,\tau ),H^2(\Omega ))\cap H^1((0,\tau ), L^2(\Omega )).
\]
It is well known that for any $f\in L^2(\Omega )$, $g\in L^2(0,\tau )$ and $q\in L^\infty (\Omega )$, the IBVP \eqref{3.1} has a unique solution 
\[u:=\mathscr{S}_q(f,g)\in H^{2,1}(Q).\] 
Moreover,  there exist a constant $C'=C'(\Omega ,m)>0$ so that, for any $q\in mB_{L^\infty (\Omega )}$,
\begin{equation}\label{3.2}
\|\mathscr{S}_q(f,g)\|_{H^{2,1}(Q)}\leq C'\|g\|_{L^2((0,\tau ))}\|f\|_{L^2(\Omega )}.
\end{equation}
We refer to \cite[Theorem 1.43, page 27]{choulli} and references therein for the statement of these results in the case of a general parabolic IBVP with non zero initial and boundary conditions.

\smallskip
Under the additional assumption that $g\in H^1(0,\tau )$, it is not hard to check that $\partial _tu$ is the solution of the IBVP \eqref{3.1} with $g$ substituted by $g'$. Hence $\partial _t \mathscr{S}_q(f,g)\in H^{2,1}(Q)$ and
\begin{equation}\label{3.3}
\|\partial _t\mathscr{S}_q(f,g)\|_{H^{2,1}(Q)}\leq C'\|g'\|_{L^2((0,\tau ))}\|f\|_{L^2(\Omega )},
\end{equation}
uniformly in $q\in mB_{L^\infty (\Omega )}$, where $C'$ is the same constant as in \eqref{3.2}.

\smallskip
As in the introduction, $\Upsilon$ is a non empty open subset of $\Gamma$ and $\Lambda =\Upsilon \times (0,\tau )$. In light of the preceding analysis $\partial _\nu \mathscr{S}_q(f,g)$ is well defined as an element of $H^1((0,\tau );L^2(\Upsilon))$ and
\[
\|\partial _\nu \mathscr{S}_q(f,g)\|_{H^1((0,\tau );L^2(\Upsilon))}\le C''\|g\|_{H^1((0,\tau ))}\|f\|_{L^2(\Omega )},
\]
for some constant $C''=C''(m,\Omega )$ uniformly in $q\in mB_{L^\infty (\Omega )}$.

\smallskip
The  following interpolation inequality will be useful in the sequel.

\begin{theorem}\label{theorem3.1}
There exist two constants $c>0$ and  $C>0$ so that, for any $q\in mB_{L^\infty (\Omega )}$, $f\in H_0^1(\Omega )$ and $g\in H^1((0,\tau ))$ with $g(0)\ne 0$,
\begin{equation}\label{3.4}
C\|f\|_{L^2(\Omega )}\leq \frac{1}{\sqrt{\epsilon}}\|f\|_{H_0^1(\Omega )}+\frac{1}{|g(0)|}e^{\tau\frac{\|g'\|_{L^2((0,\tau))}^2}{|g(0)|^2}}e^{c\epsilon}\|\partial _\nu \mathscr{S}_q(f,g)\|_{H^1((0,\tau );L^2(\Upsilon))},\;\; \epsilon \ge 1.
\end{equation}
\end{theorem}

\begin{proof}
Pick $f\in H_0^1(\Omega )$ and $q\in mB_{L^\infty (\Omega )}$. Without loss of generality, we may assume that $q\geq 0$. Indeed, we have only to substitute $u$ by $ue^{-m t}$, which is the solution of the IBVP \eqref{3.1} when $q$ is replaced by $q+m\in 2mB_{L^\infty (\Omega )}$.

\smallskip
Let $v:=\mathcal{S}_q(f)\in H^{2,1}(Q)$ be the unique solution of the IBVP
\begin{equation*}
\left\{
\begin{array}{lll}
 \partial _t v - \Delta v + q(x)v  = 0 \;\; &\mbox{in}\;   Q, 
 \\
v = 0 &\mbox{on}\;  \Sigma, 
\\
v(\cdot ,0) = f.
\end{array}
\right.
\end{equation*}
Then $\partial _\nu \mathcal{S}_q(f)$ is well defined as an element of $L^2(\Lambda )$. As for the wave equation
\[
\partial _\nu \mathscr{S}_q(f,g)_{|\Lambda }(\cdot ,t)=\int_0^t g(t-s)\partial _\nu \mathcal{S}_q(f)_{|\Lambda}(\cdot ,s)ds.
\]
Therefore
\begin{equation}\label{3.5}
\|\partial _\nu \mathcal{S}_q(f)\|_{L^2(\Lambda )}\le \frac{\sqrt{2}}{|g(0)|}e^{\tau\frac{\|g'\|_{L^2((0,\tau))}^2}{|g(0)|^2}}\|\partial _\nu \mathscr{S}_q(f,g)\|_{H^1((0,\tau );L^2(\Upsilon ))}.
\end{equation}

On the other hand, as it is shown in \cite{AC1}, the following final time observability inequality holds
\begin{equation}\label{3.6}
\|\mathcal{S}_q(f)(\cdot ,\tau )\|_{L^2(\Omega )}\le K\|\partial _\nu \mathcal{S}_q(f)\|_{L^2(\Lambda )},
\end{equation}
for some constant $K>0$, independent on $q$ and $f$.

\smallskip
A combination of \eqref{3.5} and \eqref{3.6} implies
\begin{equation}\label{3.7}
C\|\mathcal{S}_q(f)(\cdot ,\tau )\|_{L^2(\Omega )}\le \frac{1}{|g(0)|}e^{\tau\frac{\|g'\|_{L^2((0,\tau))}^2}{|g(0)|^2}}\|\partial _\nu \mathscr{S}_q(f,g)\|_{H^1((0,\tau );L^2(\Upsilon ))}.
\end{equation}

Denote by $0 <\lambda _1\le \lambda _2 \le  \ldots \le \lambda _k \ldots \rightarrow +\infty$ the sequence of eigenvalues of  the unbounded operator $A:L^2(\Omega )\rightarrow L^2(\Omega )$ given by
\[ A=-\Delta +q, \;\;D(A)=H_0^1(\Omega )\cap H^2(\Omega ).\]
Let $(\phi _k)$ be a sequence of eigenfunctions, each $\phi_k$ corresponds to $\lambda _k$, so that $(\phi _k)$ form an orthonormal basis of $L^2(\Omega )$. 

\smallskip
A usual spectral decomposition yields
\[
\mathcal{S}_q(f)(\cdot ,\tau )=\sum_{\ell \geq 1}e^{-\lambda _k\tau}(f,\phi_\ell )_{L^2(\Omega )}\phi_\ell .
\]
Here $(\cdot ,\cdot )_{L^2(\Omega )}$ is the usual scalar product on $L^2(\Omega )$. In particular
\[
(f,\phi_\ell )_{L^2(\Omega )}^2\le e^{2\lambda _\ell \tau} \|\mathcal{S}_q(f)(\cdot ,\tau )\|_{L^2(\Omega )}^2,\;\; \ell \ge 1.
\]
Whence, for any integer $N\geq 1$,
\[
\sum_{\ell =1}^N(f,\phi_\ell )_{L^2(\Omega )}^2 \le Ne^{2\lambda _N\tau}\|\mathcal{S}_q(f)(\cdot ,\tau )\|_{L^2(\Omega )}^2.
\]
This and the fact that $\left( \sum_{\ell\ge 1}\lambda _\ell (\cdot ,\phi_\ell )_{L^2(\Omega )}^2\right)^{1/2}$ is an equivalent norm on $H_0^1(\Omega )$ lead
\begin{align*}
\|f\|_{L^2(\Omega )}^2 &= \sum_{\ell =1}^N(f,\phi_\ell )_{L^2(\Omega )}^2+\sum_{\ell \ge N+1}^N(f,\phi_\ell )_{L^2(\Omega )}^2
\\
&\le \sum_{\ell =1}^N(f,\phi_\ell )_{L^2(\Omega )}^2+\frac{1}{\lambda_{N+1}}\sum_{\ell \ge N+1}^N\lambda_\ell (f,\phi_\ell )_{L^2(\Omega )}^2
\\
& \le Ne^{2\lambda _N\tau}\|\mathcal{S}_q(f)(\cdot ,\tau )\|_{L^2(\Omega )}^2+ \frac{1}{\lambda_{N+1}}\|f\|_{H^1(\Omega )}.
\end{align*}
In light of \eqref{1.8}, this estimate gives
\begin{equation}\label{3.8}
\|f\|_{L^2(\Omega )}^2\leq Ne^{2\widetilde{c}\lambda _N^{2/n}\tau}\|\mathcal{S}_q(f)(\cdot ,\tau )\|_{L^2(\Omega )}^2+ \frac{\widetilde{c}}{(N+1)^{2/n}}\|f\|_{H^1(\Omega )}.
\end{equation}

Let $\epsilon \ge 1$ and $N\ge 1$ be the unique integer so that $N\le \epsilon^{n/2}<N+1$. We obtain in a straightforward manner from \eqref{3.8}

\[
\|f\|_{L^2(\Omega )}^2\leq e^{(2\widetilde{c}\tau +1)\epsilon}\|\mathcal{S}_q(f)(\cdot ,\tau )\|_{L^2(\Omega )}^2+ \frac{\widetilde{c}}{\epsilon}\|f\|_{H^1(\Omega )}^2.
\]
The proof is completed by using the elementary inequality $\sqrt{a+b}\leq \sqrt{a}+\sqrt{b}$, $a,b\ge 0$ and \eqref{3.7}.
\end{proof}

When $q\in L^\infty (\Omega )$ and $g\in H^1((0,\tau ))$ satisfying $g(0)\ne 0$ are fixed, we set $\mathscr{S}(f):=\mathscr{S}_q(f,g)$. In that case \eqref{3.4} takes the simple form, 
\[
\widetilde{C}\|f\|_{L^2(\Omega )}\leq \frac{1}{\sqrt{\epsilon}}\|f\|_{H_0^1(\Omega )}+e^{c\epsilon}\|\partial _\nu \mathscr{S}(f) \|_{H^1((0,\tau );L^2(\Upsilon))},\;\; \epsilon \ge 1,
\]
for any $f\in H_0^1(\Omega )$. Of course the constant $\widetilde{C}$ depends on $q$ and $g$.

\smallskip
The last estimate enables us to get a logarithmic stability estimate for the inverse source problem consisting in determining $f$ from $\partial _\nu \mathscr{S}(f){_{|\Lambda}}$.

\begin{corollary}\label{corollary3.1}
Fix $q\in L^\infty (\Omega )$, $g\in H^1((0,\tau ))$ satisfying $g(0)\ne 0$. There exists a constant $\widehat{C}=\widehat{C}(n,\Omega, q,g,\Upsilon ,m)>0$ so that, for any $f\in mB_{H_0^1(\Omega )}$,
\[
\widehat{C}\|f\|_{L^2(\Omega )}\le \Phi \left(\|\partial _\nu \mathscr{S}(f) \|_{H^1((0,\tau );L^2(\Upsilon))}\right),
\]
where $\Phi (\gamma )=\left| \ln \gamma \right|^{-1/2}+\gamma$, $\gamma >0$.
\end{corollary}

Next, we consider the IBVP
\begin{equation}\label{3.9}
\left\{
\begin{array}{lll}
 \partial _t u - \Delta u + q(x)u  = 0 \;\; &\mbox{in}\;   Q, 
 \\
u = 0 &\mbox{on}\;  \Sigma, 
\\
u(\cdot ,0) = u_0.
\end{array}
\right.
\end{equation}
To any $q\in L^\infty (\Omega )$ and $u_0\in H_0^1(\Omega )$ corresponds a unique solution $u:=\mathbf{S}_q(u_0)\in H^{2,1}(Q)$ and
\begin{equation}\label{3.10}
\|\mathbf{S}_q(u_0)\|_{H^{2,1}(Q)}\le C' \|u_0\|_{H_0^1(\Omega )},
\end{equation}
uniformly in $q\in mB_{L^\infty (\Omega )}$, where $C'$ is the same constant as in \eqref{3.2}.

\smallskip
Let $\mathcal{H}_0=\{w\in H_0^1(\Omega );\; \Delta w\in H_0^1(\Omega )\}$ that we equip with its natural norm
\[
\|u\|_{\mathcal{H}_0}=\| u\|_{H_0^1(\Omega )}+\|\Delta u\|_{H_0^1(\Omega )}.
\]

When $q\in W^{1,\infty}(\Omega )$ and $u_0\in \mathcal{H}_0$ then it is straightforward to check that $\partial _t\mathbf{S}_q(u_0)=\mathbf{S}_q(\Delta u_0+qu_0)$. So applying \eqref{3.10}, with $u_0$ substituted by $\Delta u_0-qu_0$, we get
\begin{equation}\label{3.11}
\|\partial _t\mathbf{S}_q(u_0)\|_{H^{2,1}(Q)}\le C' \|u_0\|_{\mathcal{H}_0},
\end{equation}
uniformly in $q\in mB_{W^{1,\infty}(\Omega )}$.

\smallskip
Bearing in mind that the trace operator $w\in H^{2,1}(Q)\mapsto \partial _\nu w\in L^2(\Lambda )$ is bounded, we obtain that $\partial _\nu \mathbf{S}_q(u_0)\in H^1((0,\tau );L^2(\Upsilon ))$ if $u_0\in \mathcal{H}_0$ and $q\in W^{1,\infty}(\Omega )$, and using \eqref{3.10} and \eqref{3.11}, we get
\[
\|\partial _\nu \mathbf{S}_q(u_0)\|_{H^1((0,\tau );L^2(\Upsilon ))}\le C_0 \|u_0\|_{\mathcal{H}_0},
\]
uniformly in $q\in mB_{W^{1,\infty}(\Omega )}$, for some constant $C_0=C_0(n,\Omega ,\tau ,m)$.

\smallskip
In other words, we proved that the operator $\mathcal{N}_q:u_0\in \mathcal{H}_0\mapsto \partial _\nu \mathbf{S}_q(u_0)\in H^1((0,\tau );L^2(\Upsilon ))$ is bounded and
\[
\|\mathcal{N}_q\|_{\mathscr{B}(\mathcal{H}_0,H^1((0,\tau );L^2(\Upsilon )))}\le C_0 ,
\]
uniformly in $q\in mB_{W^{1,\infty}(\Omega )}$.

\smallskip
From here on, for sake of simplicity, the norm of $\mathcal{N}_q-\mathcal{N}_{q_0}$ in $\mathscr{B}(\mathcal{H}_0;H^1((0,\tau );L^2(\Upsilon))$ is simply denoted by $\|\mathcal{N}_q-\mathcal{N}_{q_0}\|$.

\begin{theorem}
There exists a constant $C>0$ so that, for any $q_0,\, q\in mB_{W^{1,\infty}(\Omega )}$,
\[
C\|q-q_0\|_{L^2(\Omega )}\le \Theta \left(\|\mathcal{N}_q-\mathcal{N}_{q_0}\|\right).
\]
Here $\Theta (\gamma )=\left| \ln \gamma \right|^{-\frac{1}{1+4n}}+\gamma $.
\end{theorem}

\begin{proof} 
Let $q_0,\, q\in mB_{W^{1,\infty}(\Omega )}$. As before, without loss of generality, we assume that $q_0\ge 0$. 

\smallskip
Let $A_0:L^2(\Omega )\rightarrow L2(\Omega)$ be the unbounded operator given by $A_0=-\Delta +q_0$ and $D(A_0)=H_0^1(\Omega )\cap H^2(\Omega )$. Denote by $0 <\lambda _1\le \lambda _2 \le  \ldots \le \lambda _k \ldots \rightarrow +\infty$ the sequence of eigenvalues of  the operator $A_0$, and $(\phi _k)$ a sequence of the corresponding eigenfunctions so that $(\phi _k)$ form an orthonormal basis of $L^2(\Omega )$. 

\smallskip
Taking into account that $\mathbf{S}_{q_0}(\phi _k)=e^{-\lambda _kt}\phi _k$, we obtain
\[
\mathbf{S}_q(\phi_k)-\mathbf{S}_{q_0}(\phi_k)=\mathscr{S}_q((q-q_0)\phi _k,e^{-\lambda_kt}).
\]
Therefore
\[
\mathcal{N}_q(\phi_k)-\mathcal{N}_{q_0}(\phi_k)=\partial _\nu \mathscr{S}_q((q-q_0)\phi _k,e^{-\lambda_kt}).
\]
Hence, a similar argument as in the preceding section yields
\[
\| \partial _\nu \mathscr{S}_q((q-q_0)\phi _k,e^{-\lambda_kt})\|_{H^1((0,\tau );L^2(\Upsilon ))}\le C\lambda_k^{3/2}  \|\mathcal{N}_q-\mathcal{N}_{q_0}\|
\]
which, in combination with estimate \eqref{3.4}, implies, with $\epsilon \geq 1$ is arbitrary,
\begin{equation}\label{3.12}
C|\Omega |^{-1/2}|(q-q_0,\phi_k)_{L^2(\Omega )}\le C\|(q-q_0)\phi_k\|_{L^2(\Omega )}\le\frac{\sqrt{\lambda_k}}{\sqrt{\epsilon}}+e^{\tau \lambda _k^2}e^{c\epsilon}\lambda _k^2 \|\mathcal{N}_q-\mathcal{N}_{q_0}\|,
\end{equation}
where we used the estimate $\|(q-q_0)\phi_k\|_{H_0^1(\Omega )}\leq C\sqrt{\lambda _k}$.

\smallskip
A straightforward consequence of estimate \eqref{3.12} is
\begin{equation}\label{3.13}
C\sum_{k=1}^N |(q-q_0,\phi_k)_{L^2(\Omega )}^2\le \frac{N\lambda _N}{\epsilon}+ Ne^{(2\tau +1)\lambda _N^2}e^{c\epsilon}\|\mathcal{N}_q-\mathcal{N}_{q_0}\|^2,
\end{equation}
for any arbitrary integer $N\ge 1$.

We pursue similarly to the proof of Theorem \ref{theorem3.1} in order to get, for arbitrary $s\geq 1$,
\[
C\|q-q_0\|_{L^2(\Omega )}^2\le \frac{s^{1+2/n}}{\epsilon}+\frac{1}{s^{2/n}}+e^{\varrho s^{1+4/n}}e^{c\epsilon}\|\mathcal{N}_q-\mathcal{N}_{q_0}\|^2.
\]
The proof is then completed in the same manner to that of Theorem \ref{theorem3.1}. 
\end{proof}


\end{document}